\documentclass[11pt,leqno,oneside,letterpaper]{amsart}
\usepackage[letterpaper,left=1.75truein, top=1truein]{geometry}
\usepackage{amsmath,amsfonts,amssymb,amscd,amsthm,amsbsy,epsf,graphics}

\newtheorem*{thm:CoprimeToOrderpreserving}{Theorem \ref{T:CoprimeToOrderpreserving}}   

\usepackage{color}
\usepackage[colorlinks,linkcolor=blue,citecolor=blue]{hyperref}
\usepackage{epsfig}
\usepackage{graphicx}
\usepackage[utf8]{inputenc}
\usepackage{tikz}

\usepackage{todonotes}
\usepackage{enumitem}

\newtheorem{theorem}{Theorem}[section]
\newtheorem{lemma}[theorem]{Lemma}
\newtheorem{proposition}[theorem]{Proposition}
\newtheorem{definition}[theorem]{Definition}

\newtheorem{corollary}[theorem]{Corollary}
\theoremstyle{definition} 

\newtheorem{remark}[theorem]{Remark}
\newtheorem{example}[theorem]{Example}

\newcommand{\Z}{\mathbb{Z}}

\newtheoremstyle{cases}
  {12pt plus 6 pt}
  {2pt}
  {\bfseries}   
  {}
  {\bfseries}
  {.}
  {.5em}
  {}
\theoremstyle{cases}

\numberwithin{subcase}{case} \numberwithin{subsubcase}{subcase}
\numberwithin{equation}{subsection}
\usepackage{tikz}
\usetikzlibrary{matrix,arrows}

\title{Ordered bases, order-preserving automorphisms and bi-orderable link groups}

\begin{document}
\author[T.~Cai]{Tommy Wuxing Cai}

\address{Department of Mathematics, University of
	Manitoba, Winnipeg, MB, R3T 2N2}
\email{cait@myumanitoba.ca}

\author[A.~Clay]{Adam Clay}

\address{Department of Mathematics,, University of
	Manitoba, Winnipeg, MB, R3T 2N2}
\email{Adam.Clay@umanitoba.ca}

\author[D.~Rolfsen]{Dale Rolfsen}

\address{Department of Mathematics, University of British Columbia, Vancouver, BC, Canada V6T 1Z2}
\email{rolfsen@math.ubc.ca}

\thanks{Adam Clay was partially supported by NSERC grant RGPIN-05343-2020 }

 \subjclass[2010]{06F15, 20F60, 57M05, 57K30}

\begin{abstract} 
We give a new criterion which guarantees that a free group admits a bi-ordering that is invariant under a given automorphism.  As an application, we show that the fundamental group of the ``magic manifold" is bi-orderable, answering a question of Kin and Rolfsen.
\end{abstract}

\maketitle

\section{Introduction}

This work was motivated by a study of the Artin action \cite{Artin47} of the braid groups $B_n$ upon the free group $F_n$ of rank $n$, as in 
\cite{KR18}. In the course of our investigation, we introduce a new method which may be applied more generally to establish that an automorphism of a free group preserves some bi-ordering.  From this, we develop new methods of determining when a link $L$ in $S^3$, thought of as the closure of a braid together with the braid axis, yields a bi-orderable link group $\pi_1(S^3 \setminus L)$.  Some definitions are in order.

If the elements of a group $G$ can be given a strict total ordering $<$ with the property that $f < g$ implies $hf < hg$ for all $f, g, h \in G$, then we say $<$ is a \textbf{left-ordering} of $G$ and that $G$ is \textbf{left-orderable}. However if we also have
$f < g$ implies $fh < gh$ for all $f, g, h \in G$, then $<$ is called a 
\textbf{bi-ordering} of $G$, and $G$ is said to be \textbf{bi-orderable}.  For example, torsion-free abelian groups and also nonabelian free groups are bi-orderable.  
Let $\varphi:G \to G$ be an automorphism. If there is a bi-ordering $<$ of $G$ such that $f<g$ if and only if $\varphi(f) < \varphi(g)$ for all $f, g \in G$, we say that $\varphi$ is \textbf{order-preserving}, that  $<$ is \textbf{invariant} under $\varphi$, and that $\varphi$ \textbf{preserves} $<$. 

 Recall that if $\sigma:I \rightarrow I$ is a bijection, then the orbit of $x \in I$ under $\sigma$ is  $\mathcal{O} = \{ \sigma^n(x) \mid n \in \mathbb{Z} \}.$ We say that an orbit $\mathcal O$ of $\sigma$ is finite (respectively infinite) if the cardinality $|\mathcal O|$ is finite (respectively infinite).

Our main theorem for showing that an automorphism of a free group $\varphi :F \to F$ is order-preserving is the following. 

\begin{thm:CoprimeToOrderpreserving}
Let $F$ be the free group generated by $\{x_i \mid i\in I\}$ where $I$ is a nonempty set.  Let $\varphi : F \to F$ be an automorphism satisfying
\[ \varphi(x_i) = w_ix_{\sigma(i)}w_i^{-1}, \forall i \in I
\]
where $w_i \in F$ and $\sigma$ is a bijection on $I$. Assume that $\sigma(i_0) = i_0$ for some $i_0\in I$ and let $h:F\to \mathbb{Z}$ be the homomorphism defined by $h(x_{i_0}) = 1$ and $h(x_j) =0$ for all $j \neq i_0$. Assume that for each finite orbit $\mathcal O$ of $\sigma$ we have $$\gcd\Big(|\mathcal{O}|,\sum_{i \in \mathcal{O}} h(w_i)\Big)=1.$$ Then there exists a bi-ordering of $F$ invariant under $\varphi$.
\end{thm:CoprimeToOrderpreserving}

One source of automorphisms of $F_n$ having a formula as in Theorem \ref{T:CoprimeToOrderpreserving} is the Artin action of $B_n$ on $F_n$ \cite{Artin47}. The braid group $B_n$ is generated by $\sigma_1,\dotsc,\sigma_{n-1}$, subject to the relations:
\begin{align*}
&\sigma_i\sigma_{i+1}\sigma_i=\sigma_{i+1}\sigma_i\sigma_{i+1}, \text{ for } 1\leq i\leq n-2\\
&\sigma_i\sigma_j=\sigma_j\sigma_i, \text{ for }|i-j|>1.
\end{align*}

 There is an embedding of $B_n$ into the automorphism group of the free froup $F_n$ on generators $\{x_1, \dots, x_n\}$, which yields the Artin action of $B_n$ on $F_n$. The embedding sends each generator $\sigma_i$ of $B_n$ to the automorphism which acts on each generator $x_j$ of $F_n$ according to 
\begin{displaymath}
   x_j^{\sigma_i} = \left\{
     \begin{array}{ll}
        x_ix_{i+1}x_i^{-1} & \mbox{ if $j = i$} \\
       x_i & \mbox{ if $j = i+1$} \\
        x_j & \mbox{ if $j \neq i, i+1$.}
     \end{array}
   \right.
\end{displaymath} 


In \cite[Section 2.4]{KR18} the following was observed.  Consider the link 
$br(\beta) = \hat{\beta} \cup A$ in $S^3$, consisting of the braid closure $\hat{\beta}$ together with the braid axis $A$.
Then the Artin action of a braid $\beta \in B_n$ on $F_n$ yields an automorphism that preserves a bi-ordering of $F_n$ if and only if the fundamental group
$\pi_1(S^3 \setminus br(\beta))$ of the complement of the link $br(\beta)$ is bi-orderable.

An unanswered question in \cite{KR18} was whether the braid
$\beta = \sigma_1^2\sigma_2^{-1} \in B_3$ yields an automorphism that preserves a bi-ordering of $F_3$.  This is equivalent to the question of whether the fundamental group of the so-called ``magic manifold"
is bi-orderable; the magic manifold is a $3$-cusped hyperbolic $3$-manifold conjectured to be of minimal volume, see Figure \ref{fig:magic_link}.  We are able to answer this question in the affirmative using Theorem \ref{T:CoprimeToOrderpreserving} (see Proposition \ref{P:magic}).  We also generate many new examples of order-preserving braids, and therefore many bi-orderable link groups whose bi-orderability could not be determined with previously known techniques (e.g. see \cite{CR12, CDN16, KR18, ito17, yamada, JJ23, Johnson_2024}).  

\begin{figure}
    \centering
    \includegraphics[scale=2]{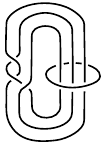}
    \caption{The link $br(\beta)$ for $\beta = \sigma_1^2\sigma_2^{-1}$, whose complement is homeomorphic to the magic manifold.}
    \label{fig:magic_link}
\end{figure}

\subsection{Organisation of the paper}

In Section \ref{S:orderedbases} we introduce generalisations of the ``positive eigenvalue" condition of \cite{PR03}, and use it to show certain automorphisms of infinitely generated abelian groups preserve a bi-ordering. In Section \ref{S:freegroups} we bootstrap this result to prove our main theorem. In Section \ref{section:links} we apply our main theorem to analyse the Artin action on $F_n$ and show that the fundamental group of the magic manifold is bi-orderable.

\subsection{Acknowledgments}

The second author would like to thank the CRM and CIRGET for their hospitality and for organizing the 2023 thematic semester on geometric group theory, during which the ideas of this paper first took shape. 

\section{Ordered bases, modules, and order-preserving homomorphisms}
\label{S:orderedbases}
Let $(S,<)$ be a set with a strict total ordering, and $f:S\rightarrow S$ a function. We say that $f$ preserves the ordering $<$ if $a<b$ implies that $f(a)<f(b)$ for all $a, b \in S$, in this case we say that $<$ is invariant under $f$, and that the function $f$ is order-preserving.  Note that our definition of order-preserving is therefore context-sensitive: A bijection $f:S \rightarrow S$ from a set $S$ to itself is order-preserving if there is a strict total ordering $<$ of $S$ preserved by $f$, whereas an automorphism $\varphi :G \rightarrow G$ of a group is order-preserving if there is a bi-ordering $<$ of $G$ preserved by $\varphi$.

In what follows, we let $R$ be the ring $\mathbb Z$, $\mathbb Q$ or $\mathbb R$, each equipped with its usual ordering. Given a free $R$-module $M$, an \textbf{ordered basis} for $M$ is a pair $ (\{v_i \mid i\in I\}, \prec)$ where $\{v_i \mid i\in I\}$ is a basis for the $R$-module $M$, and $\prec$ is a strict total ordering of the the index set $I$.

\begin{definition}\label{D:LexicographicOrderingAssociatedWithB}
Given a free $R$-module $M$ and an ordered basis $\mathcal{B} = (\{v_i  \mid  i\in I\}, \prec)$, the \textbf{lexicographic ordering of $M$ associated with $\mathcal{B}$} is the bi-ordering $<_{\mathcal{B}}$ of the abelian group $M$ defined as follows:
Given $\alpha=\sum_i a_iv_i \in M \setminus \{0\}$, set $i(\alpha) = \min_{\prec}\{ j\in I \mid a_j\neq0\}$ and define $0<_{\mathcal{B}} \alpha$ if only if $a_{i(\alpha)}>0$.
\end{definition}

 \begin{definition}\label{D:PTriangular}
 Let $M$ be a free $R$-module, $f:M\rightarrow M$ an $R$-module homomorphism, and $\mathcal{B} = (\{v_i \mid i\in I\}, \prec)$ an ordered basis. We say that $f$ is \textbf{positively triangular with respect to $\mathcal{B}$} if there is an order-preserving map $\eta:I\rightarrow I$ with respect to $\prec$ such that for each $i\in I$,
\begin{align*}
f(v_i)=\lambda_iv_{\eta(i)}+\sum_{j\succ\eta(i)}c_{i,j}v_j,
\end{align*}
where $\lambda_i, c_{i,j}\in R$ and $\lambda_i>0$.  We say that $f$ is \textbf{positively triangular} if there exists an ordered basis $\mathcal{B}$ such that $f$ is positively triangular with respect to $\mathcal{B}$.
\end{definition}

\begin{remark}\label{R:DefOfPositiveTriangular}Note that if $M$ is a finite rank $R$-module, say with ordered basis $\mathcal{B} = \{v_1, \ldots, v_n\}$ where the index set $\{1, \dots, n\}$ is equipped with its natural ordering, then $f : M \rightarrow M$ is positively triangular with respect to $\mathcal{B}$ if and only if $f$ is represented relative to $\mathcal B$ by an upper triangular matrix with positive entries on the diagonal.\end{remark}

\begin{lemma}\label{L:PositiveTriangularbi-ordering}
Let $M$ be a free $R$-module and $\mathcal{B} = (\{v_i \mid i\in I\}, \prec)$ an ordered basis for $M$.  If  $f:M\rightarrow M$ is an $R$-module homomorphism that is positively triangular with respect to $\mathcal{B}$, then $f$ preserves the bi-ordering $<_{\mathcal{B}}$.
\end{lemma}
\begin{proof}
Suppose $f$ is positively triangular with respect to $\mathcal{B}$ and let $\prec$,  $\eta$ , $\lambda_i$ and  $c_{i,j}$ be as given in Definition \ref{D:PTriangular}.  Recall that if $\alpha=\sum_i a_iv_i \in M \setminus \{0\}$ then $i(\alpha) := \min_{\prec}\{ j\in I \mid a_j\neq0\}$ . To check that $<_{\mathcal{B}}$ is invariant under $f$, assume that $\alpha=\sum_i a_iv_i>_{\mathcal{B}}0$ and we want to prove that $f(\alpha)>_{\mathcal{B}}0$.  We have:
\begin{align*}
f(\alpha)&=a_{i(\alpha)}f(v_{i(\alpha)})+\sum_{j\succ i(\alpha)}a_jf(v_j)\\
&=a_{i(\alpha)}\left(\lambda_{i(\alpha)}v_{\eta(i(\alpha))}+
\sum_{j\succ\eta(i(\alpha))}c_{i(\alpha),j}v_j\right)\\
&\qquad \qquad +\sum_{j\succ i(\alpha)}a_j\left(\lambda_{j}v_{\eta(j)}+\sum_{k\succ \eta(j)}c_{j,k}v_k\right)\\        &=a_{i(\alpha)}\lambda_{i(\alpha)}v_{\eta(i(\alpha))}+\sum_{k\succ \eta(i(\alpha))}c_{\alpha,k}'v_k.
\end{align*}
To reach the last equality, we have regrouped terms to arrive at new coefficients $c_{\alpha,k}'$ for $k\succ \eta(i(\alpha))$, where we used the fact that $k\succeq\eta(j)$ and $j\succ i(\alpha)$ implies $k\succ\eta(i(\alpha))$, as $\succ$ is invariant under $\eta$.

Since both $a_{i(\alpha)}$ and $\lambda_{i(\alpha)}$ are positive, we conclude that $f(\alpha)>_{\mathcal{B}} 0$, as desired.
\end{proof}
\begin{definition}\label{D:TensorProductOfOrderedBases} Let $\mathcal{B} = (\{v_i  \mid  i\in I\}, \prec)$  and $\mathcal C=(\{w_j\mid j\in J\},\lhd)$  be ordered bases of $R$-modules of $M$ and $N$ respectively. We define the \textbf{tensor} $\mathcal B\otimes \mathcal C$  to be the ordered basis $(\{v_i\otimes w_j \mid (i,j)\in I\times J\},<)$ of $R$-module $M\otimes N$, where the total  ordering $<$ on $I\times J$ is defined lexicographically: $(i,j)<(s,t)$ if $i\prec s$ or $i=s$ and $j\lhd t$.
\end{definition}
\begin{lemma}\label{L:TensorOfPositivelyTrangularHom}
Let  $f:M\rightarrow M$ and $g:N\rightarrow N$ be $R$-module homomorphisms. Assume that $f$ and $g$ are positively triangular with respect to ordered bases $\mathcal B$  and $\mathcal C$ respectively. Then the $R$-module homomorphism $f\otimes g: M\otimes N\rightarrow M\otimes N$ is positively triangular with respect to $\mathcal B\otimes \mathcal C$. 
\end{lemma}
\begin{proof} Let $\mathcal{B} = (\{v_i  \mid  i\in I\}, \prec)$  and $\mathcal C=(\{w_j\mid j\in J\},\lhd)$ .
For $f$, let  $\eta :I \rightarrow I $ be as in Definition \ref{D:PTriangular}, and for $g$, we will use the function $\xi: J \rightarrow J$.  Define the map $\sigma:I\times J\rightarrow I\times J$: $\sigma(i,j)=(\eta(i),\xi(j))$, which is clearly preserves the lexicographic ordering $<$ of $I \times J$.  Writing $u_{(i,j)}$ for $v_i\otimes w_j$,  we have
\begin{align*}
 &(f\otimes g)(u_{(i,j)})=(f\otimes g)(v_i\otimes w_j)=f(v_i)\otimes g(w_j)\\
=&\left(\lambda_iv_{\eta(i)}+\sum_{s\succ\eta(i)}a_{i,s}v_s\right)\otimes\left(\mu_jw_{\xi(j)}+\sum_{t\rhd\xi(j)}b_{j,t}w_t\right)\\
&=\lambda_i\mu_ju_{(\eta(i),\xi(j))}+\sum_{t\rhd\xi(j)}\lambda_ib_{j,t}u_{(\eta(i),t)}\\
&\qquad\qquad\qquad+\sum_{s\succ\eta(i)}a_{i,s}\mu_ju_{(s,\xi(j))}\\
&\qquad\qquad\qquad+\sum_{s\succ\eta(i),t\rhd\xi(j)}a_{i,s}b_{j,t}u_{(s,t)}\\
&=\lambda_i\mu_j u_{\sigma(i,j)}+\sum_{(s,t)>\sigma(i,j)}c_{i,j,s,t}u_{(s,t)},
\end{align*}
for some $c_{i,j,s,t} \in R$.  As $\lambda_i\mu_j >0$,  $f\otimes g$ is thus positively triangular with respect to $\mathcal B\otimes \mathcal C$. 
\end{proof}

\section{Order-preserving automorphisms of free groups}
\label{S:freegroups}

 We now aim to build upon the results of the last section, bootstrapping to yield order-preserving automorphisms of free groups of arbitrary rank.  
 
 Our first lemma is a generalisation of \cite[Lemma 4.5]{PR03}, wherein we confirm that the free group $F$ in that proof need not be finitely generated. The proof remains largely unchanged, with two exceptions: (1) the summations appearing below over the variable $j$ are sums of terms drawn from an infinite set (though each summation is still finite), and (2) the equation  \eqref{E:AbelianizationOfOneElementUsingDis} is demonstrated directly without referring to the Jacobian matrix as in \cite{PR03}.  Despite the changes being minor, we present the proof here for the sake of completeness, and so that the reader can confirm for themselves that the conclusion of the lemma is independent of the cardinality of the generating set of $F$.

Here is the setup for the lemma. Let $F$ be a free group generated by $\{z_i \mid i\in J\}$, where $J$ is not necessarily finite (unlike \cite[Lemma 4.5]{PR03}). Let $F_k$ be the $k$-th term of the lower central series of $F$, defined by $F_1=F$ and $F_{k+1} = [F_k, F]$.  Let $H$ be the abelianisation $F/[F,F]$ of $F$, written additively. Let $\mathbb ZF$ be the group ring of $F$ over $\mathbb Z$, and $\epsilon:\mathbb ZF\rightarrow \mathbb{Z}$  the homomorphism of rings defined by $\epsilon(\sum_{i=1}^nk_ig_i)=\sum_{i=1}^nk_i$ for $k_i\in\mathbb Z$ and $g_i\in F$. Set $I = \ker(\epsilon)$,  which is an ideal of the ring $\mathbb ZF$. For $g\in F$ and $k\in\{1,2,\dotsc\}$, we have $g\in F_k$ if and only if $g-1\in I^k$ \cite[First paragraph of Page 557]{Fox53}. Therefore, there is an injective homomorphism $\psi_k:F_k/F_{k+1}\rightarrow I^k/I^{k+1}$, sending $[g]$ to $[g-1]$, where the square brackets are to be interpreted as equivalence classes in the appropriate quotient. Note that a basis of  $I^k/I^{k+1}$ is $$\Big\{[(z_{j_1}-1)\dotsm(z_{j_k}-1)] \mid j_1,\dotsc,j_k\in J\Big\}.\footnote{This generates the group because of the Taylor expansion  formula for the Fox derivation. To prove that it is $\mathbb Z$-independent, we use the the fact that $D_{i_1}\dotsm D_{i_k}((z_{j_1}-1)\dotsm (z_{j_k-1}-1))\neq0$ if and only if $(i_1,\dotsc,i_k)=(j_1,\dotsc,j_k)$, which can easily be proved by induction.}$$  Hence we can define a $\mathbb Z$-module isomorphism $i_k:I^k/I^{k+1}\rightarrow H^{\otimes k}$ by $$i_k([(z_{j_1}-1)\dotsm(z_{j_k}-1)]) = a_{j_1}\otimes\dotsm\otimes a_{j_k},$$ where $a_j:=[z_j]$ is the image of $z_j$ in $H$. Moreover, there is an injective $\mathbb Z$-module homomorphism $j_k$ from $H^{\otimes k}$ to $(H\otimes\mathbb R)^k$, given by
$$j_k( a_1\otimes \dotsm \otimes a_k)= (a_1\otimes 1)\otimes\dotsm\otimes(a_k\otimes 1), \forall a_1,\dotsc,a_i\in H.$$

\begin{lemma}\label{L:Commutative}
For every automorphism $\varphi :F \rightarrow F$ and  every positive integer $k$, the following diagram commutes:
\begin{center}
 \begin{tikzpicture}[description/.style={fill=white,inner sep=2pt}]
\matrix (m) [matrix of math nodes, row sep=1em,
column sep=2em, text height=1.5ex, text depth=0.25ex]
{  F_k/F_{k+1}  & & I^k/I^{k+1}& &H^{\otimes k}&&(H\otimes\mathbb R)^{\otimes k}\\
& &&&\\
 F_k/F_{k+1}  & & I^k/I^{k+1}& &H^{\otimes k} &&(H\otimes\mathbb R)^{\otimes k}\\ };
\path[->,font=\scriptsize]
(m-1-1) edge node[auto] {$\psi_k$} (m-1-3);
\path[->,font=\scriptsize]
(m-1-3) edge node[auto] {$i_k$} (m-1-5);
\path[->,font=\scriptsize]
(m-1-5) edge node[auto] {$j_k$} (m-1-7);
\path[->,font=\scriptsize]
(m-1-1) edge node[auto] {$\varphi_k$} (m-3-1);
\path[->,font=\scriptsize]
(m-1-3) edge node[auto] {$\varphi_k'$} (m-3-3);
\path[->,font=\scriptsize]
(m-3-1) edge node[auto] {$\psi_k$} (m-3-3);
\path[->,font=\scriptsize]
(m-3-3) edge node[auto] {$i_k$} (m-3-5);
\path[->,font=\scriptsize]
(m-3-5) edge node[auto] {$j_k$} (m-3-7);
\path[->,font=\scriptsize]
(m-1-5) edge node[auto] {$\varphi_{ab}^{\otimes k}$} (m-3-5);
\path[->,font=\scriptsize]
(m-1-7) edge node[auto] {$(\varphi_{ab}\otimes Id)^{\otimes k}$} (m-3-7);
,\end{tikzpicture}
\end{center}
where the vertical maps are induced from $\varphi$.
\end{lemma}
\begin{proof} It is easy to check that the first rectangle and the last rectangle are commutative, because all the maps involved are canonical.  Therefore we focus on commutativity of the middle rectangle.
Let $\alpha:=[(z_{j_1}-1)\dotsm(z_{j_k}-1)]\in I^k/I^{k+1}$, it is enough to prove that $i_k(\varphi_k'(\alpha))=\varphi_{ab}^{\otimes k}(i_k(\alpha))$ for all such $\alpha$.

First, $\varphi_k'(\alpha)=[(\varphi(z_{j_1})-1)\dotsm(\varphi(z_{j_k})-1)]$. Note that we have
\begin{align*}
\varphi(z_{j_s})-1=\sum_j D_j^0(\varphi(z_{j_s}))(z_j-1)+O(2),
\end{align*}
where $D_j=\frac\partial{\partial z_j}$ is the derivation defined in \cite{Fox53}, $D_j^0(w)=\epsilon(D_j(w))$, and $O(2)$ is a term in $I^2$.
We compute
\begin{align}\label{E:Commutativity1}
i_k(\varphi_k'(\alpha))&=i_k\left(\Big[\big(\sum_j D_j^0(\varphi(z_{j_1}))(z_j-1)\big)\dotsm \big(\sum_j D_i^0(\varphi(z_{j_k}))(z_j-1)\big)\Big]\right)\\\nonumber
&=\big(\sum_j D_j^0(\varphi(z_{j_1}))a_j\big)\otimes\dotsm \otimes\big(\sum_j D_j^0(\varphi(z_{j_k}))a_j\big)\\\nonumber
&=\Big[\sum_j D_j^0(\varphi(z_{j_1}))z_j\Big]\otimes\dotsm \otimes\Big[\sum_j D_j^0(\varphi(z_{j_k}))z_j\Big].
\end{align}
Second, we have
\begin{align}\label{E:Commutativity2}
\varphi_{ab}^{\otimes k}\left(i_k(\alpha)\right)&=\varphi_{ab}^{\otimes k}\left(a_{j_1}\otimes\dotsm\otimes a_{j_k}\right)\\\nonumber
&=\varphi_{ab}(a_{j_1})\otimes\dotsm\otimes\varphi_{ab}(a_{j_k})\\\nonumber
&=[\varphi(z_{j_1})]\otimes\dotsm \otimes [\varphi(z_{j_k})].
\end{align}
Comparing \eqref{E:Commutativity1} and \eqref{E:Commutativity2}, we only need to prove that for $w\in F$, we have
\begin{align}\label{E:AbelianizationOfOneElementUsingDis}
[w]=\left[\sum_j D_j^0(w) z_j\right]
 \end{align}
 in $H$. But this is true, because $D_j^0(w)$ is the sum of the exponents of $z_j$ that appear in $w$, and recall that we write the operation in $H$ additively. \end{proof}

We are now ready to prove the main result of this section, which by Remark \ref{R:DefOfPositiveTriangular}, is a generalisation of \cite[Theorem 2.6]{PR03} from finitely generated free groups to free groups of arbitrary rank. 
\begin{proposition}\label{P:PT-Invariantbi-order}
Let $F$ be a free group, $\varphi :F \rightarrow F$ an automorphism, set $H:=F/[F,F]$, and let $\varphi_{ab}:H \rightarrow H$ be the automorphism induced by $\varphi$. If $\varphi_{ab}\otimes Id : H\otimes \mathbb R \rightarrow H\otimes \mathbb R$ is positively triangular, then there is a bi-ordering $<$ on $F$ which is invariant under $\varphi$.
\end{proposition}
\begin{proof}
Let $k$ be an arbitrary positive integer. Let $F_k$ be the $k$-th term in the lower central series of $F$ and $\varphi_k:F_k/F_{k+1} \rightarrow F_k/F_{k+1}$ the automorphism induced by $\varphi$. By Lemma \ref{L:TensorOfPositivelyTrangularHom}, $(\varphi_{ab}\otimes Id)^{\otimes k}:(H\otimes \mathbb R)^{\otimes k} \rightarrow (H\otimes \mathbb R)^{\otimes k}$ is positively triangular. By Lemma \ref{L:PositiveTriangularbi-ordering}, there is a bi-ordering $<'_k$ of $(H\otimes \mathbb R)^{\otimes k}$ which is invariant under $(\varphi_{ab}\otimes Id)^{\otimes k}$.

By Lemma \ref{L:Commutative}, the ordering $<'_k$ can be used to define a bi-ordering $<_k$ of $F_k/F_{k+1}$ that is invariant under $\varphi_k$.  The ordering is defined by declaring that for $a,b\in F_k/F_{k+1}$, we have $a<_kb$ if and only if $j_k(i_k(\psi_k(a)))<'_kj_k(i_k(\psi_k(b)))$.

Finally, we use these orderings $<_k$ ($k=1,2,\dotsc$) to define a bi-ordering $<$ on $F$. It is known that $\cap_{k=1}^\infty F_k=\{1\}$, thus for $a\neq1$ in $F$, there is a largest $k$, denoted as $k(a)$,  such that $a\in F_k$. Then we define $a>1$ if and only if $[a]>_{k(a)} 1$ in $F_{k(a)}/F_{k(a)+1}$, where $[a]$ indicates the equivalence class  in the quotient $F_{k(a)}/F_{k(a)+1}$ which contains $a$. This defines a bi-ordering of $F$ and it is invariant under $\varphi$ by construction.
\end{proof}

 Proposition \ref{P:PT-Invariantbi-order} tells us that for automorphisms $\varphi$ of a free group $F$, a sufficient condition for $\varphi$ to be order-preserving  is that $\varphi_{ab} \otimes Id$ is positively triangular. However, this condition is not a necessary condition for $\varphi$ to be order-preserving.  We will see in Example \ref{E:OrderPreservingButNotPositivelyTriangular} that there exist homomorphisms $\varphi$ such that  $\varphi_{ab}\otimes Id$ is not positively triangular, but $\varphi$ is still order-preserving.

\begin{remark}\label{R:PositiveTriangulartorProposition-WhyTimesR}
Let $\varphi$ be an automorphism of a free group $F$ and set $H=F/[F,F]$.
If $\varphi_{ab}:H \rightarrow H$ is positively triangular, then $\varphi_{ab}\otimes Id : H\otimes\mathbb R \rightarrow H\otimes\mathbb R$ is positively triangular. However, the reverse is not true, and it is for this reason that we insist $\varphi_{ab}\otimes Id : H\otimes\mathbb R \rightarrow H\otimes\mathbb R$ be positively triangular in Proposition \ref{P:PT-Invariantbi-order}.

As an example of this behaviour, let $F$ be the free group generated by $x_1,x_2$. Let $\varphi$ be the automorphism of $F$ given by $$\varphi(x_1) =  x_1^2x_2, \mbox{ and } \varphi(x_2)=  x_1^2x_2x_1x_2.$$ Writing the abelian group $H=F/[F,F]$ additively and the cosets of $x_1$ and $x_2$ as $\overline{x_1}$ and $\overline{x_2}$ respectively, $\varphi_{ab}$ sends $\overline{x_1}$  to $2\overline{x_1}+\overline{x_2}$ and $\overline{x_2}$ to $3\overline{x_1}+2\overline{x_2}$ respectively. Then matrix of $\varphi\otimes Id : H \otimes \mathbb{R} \rightarrow H \otimes \mathbb{R}$ is therefore $A=\left[\begin{smallmatrix}2&3\\1&2\end{smallmatrix}\right]$, with eigenvalues $2\pm\sqrt3$. Hence $\varphi :H \rightarrow H$ is not positively triangular, whereas $\varphi_{ab}\otimes Id : H\otimes\mathbb R \rightarrow H\otimes\mathbb R$ is positively triangular, because the matrix $A$ is diagonalisable over $\mathbb R$. Note the map $\varphi$ is thus order-preserving by Proposition \ref{P:PT-Invariantbi-order}; we can explicitly define a bi-ordering on $H\otimes\mathbb R$ which is invariant under $\varphi_{ab}\otimes Id$: $\overline{x_1}\otimes a+\overline{x_2}\otimes b>0$ if and only if $a+b\sqrt3>0$ for  $a,b\in\mathbb R$.
\end{remark}
We can check that the construction of $<$ in the proof of Proposition \ref{P:PT-Invariantbi-order} doesn't depend on the homomorphism $\varphi$, and instead depends only on the underlying choice of ordered basis. In fact, the ordering is constructed as in our next definition, where \eqref{E:KeyInequality} involves Definition \ref{D:LexicographicOrderingAssociatedWithB}, Definition \ref{D:TensorProductOfOrderedBases} and the diagram in Lemma \ref{L:Commutative}. 
\begin{definition}\label{D:bi-orderingAssociatedWithOrderedBasis}
Let $F$ be a free group and $\mathcal{B}$ be an ordered basis of $(F/[F,F])\otimes\mathbb R$.
Let $<_k$ denote the ordering of $F_k/F_{k+1}$ defined by $a<_kb$ if and only if 
\begin{align}\label{E:KeyInequality}
j_k(i_k(\psi_k(a)))<_{\mathcal B^{\otimes k}} j_k(i_k(\psi_k(b))).
\end{align}
Let $<$ be the bi-ordering of $F$ induced by the orderings $<_k$ as in the proof of Proposition \ref{P:PT-Invariantbi-order}. We call $<$ the \textbf{bi-ordering of $F$ associated with $\mathcal B$} and denote it by $<_{\mathcal B}$.
\end{definition}

We can thus, upon reviewing the proof, give a version of Proposition \ref{P:PT-Invariantbi-order} that makes explicit the invariant bi-ordering.
\begin{theorem}\label{T:PT-Invariantbi-order}
Let $F$ be a free group and $\mathcal{B}$ an ordered basis of $(F/[F,F])\otimes \mathbb R$. Let $<_{\mathcal B}$ be the bi-ordering of $F$ associated with $\mathcal B$. If $\varphi: F\rightarrow F$ is an automorphism such that $$\varphi_{ab}\otimes Id:(F/[F,F])\otimes\mathbb R \rightarrow (F/[F,F])\otimes\mathbb R$$ is positively triangular with respect to $\mathcal{B}$, then $<_{\mathcal B}$ is invariant under $\varphi$. 
\end{theorem}

As an application of Theorem \ref{T:PT-Invariantbi-order}, we are now ready to prove Theorem \ref{T:CoprimeToOrderpreserving}, our main result from the introduction. 


\begin{theorem}
\label{T:CoprimeToOrderpreserving}
Let $F$ be the free group generated by $\{x_i \mid i\in I\}$ where $I$ is a nonempty set.  Suppose that $\varphi : F \rightarrow F$ is an automorphism satisfying
\[ \varphi(x_i) = w_ix_{\sigma(i)}w_i^{-1}, \forall i\in I
\]
where $w_i \in F$ and $\sigma$ is a bijection on $I$. Assume that $\sigma$ fixes $i_0\in I$.  Define homomorphism $h:F\rightarrow \mathbb{Z}$ by $h(x_{i_0}) = 1$ and $h(x_j) =0$ for all $j \neq i_0$. Assume that for each finite orbit $\mathcal O$ of $\sigma$, we have $$\gcd\left(|\mathcal O|,\sum_{i\in\mathcal O}h(w_i)\right)=1.$$  Then there exists a bi-ordering of $F$ invariant under $\varphi$.
\end{theorem}

\begin{proof} Let $K$ be the kernel of $h$. We see that 
$\varphi(K)=K$ and thus the restriction $\varphi|_K$ is an automorphism of $K$. Let $\psi:F\rightarrow F$ be the conjugation $\alpha\rightarrow x_{i_0}\alpha x_{i_0}^{-1}$. Similarly, we can verify that $\psi|_K$ is an automorphism on $K$. We now proceed in two steps.

First, we assume that there is a bi-ordering $<_K$ on $K$ which is invariant under both $\varphi|_K$ and $\psi|_K$. Under this assumption, we now show that the lexicographic left-ordering on $F$ coming from the short exact sequence
$$1\rightarrow K\xrightarrow{i}F\xrightarrow{h} \Z\rightarrow \{0\}$$ is in fact a bi-ordering invariant under $\phi$.  Here, $i$ is the inclusion of $K$ in $F$. The positive cone of $F$ is defined to be
 $$P_{F}:=i(P_K)\cup h^{-1}(\mathbb Z_{>0})=P_K\cup \{\alpha\in F_{n} \mid h(\alpha)>0\},$$
 where $P_K$ is the positive cone of $K$ corresponding to $<_K$. To verify that it gives a bi-ordering, we only need to verify that $x_{i_0}^{k}P_Kx_{i_0}^{-k}=P_K$ for all $k\in\mathbb Z$. This is true because $x_{i_0}P_Kx_{i_0}^{-1}=\psi|_K(P_K)=P_K$ as $<_K$ is invariant under $\psi|_K$.
  To verify that it is invariant under $\varphi$, we want to show that $\varphi(P_{F})\subseteq P_{F}$. This is true, because $\varphi(P_K)=\varphi|_K(P_K)=P_K$ using that $<_K$ is invariant under $\varphi|_K$; and  if $\alpha\in F$ satisfies $h(\alpha)>0$, then have $h(\varphi(\alpha))=h(\alpha)>0$.

So, to prove the theorem it suffices to prove that there is a bi-ordering on $K$ which is invariant under both $\varphi|_{K}$ and $\psi|_{K}$, which is our second step in the proof.
    
    It can be shown that $K$ has a free basis: $$\left\{x_{i_0}^jx_ix_{i_0}^{-j}\mid i\in I\backslash\{i_0\}, j\in\mathbb Z\right\}.$$ For $\alpha\in K$, we write $\overline{\alpha}$ for the image of $\alpha$ in the quotient $H:=K/[K,K]$, which is the abelianisation of $K$. Then $H$ is a free abelian group and 
    \begin{align}\label{E:ABaseOfH}
        \mathfrak{B}:=\{A_{i,j} \mid i\in I\backslash\{i_0\}, j\in\mathbb Z\} \text{ with } A_{i,j}:=\overline{x_{i_0}^jx_ix_{i_0}^{-j}}
    \end{align} is a basis of $H$. By Theorem \ref{T:PT-Invariantbi-order} and the second sentence of Remark \ref{R:PositiveTriangulartorProposition-WhyTimesR}, we only need to prove that there is an ordered basis $\mathcal{B}$ of $H$, such that $(\varphi|_K)_{ab}$ and $(\psi|_K)_{ab}$ are positively triangular with respect to $\mathcal{B}$. We proceed by the following sub-steps:
 \begin{enumerate}
\item \underline{Claim:}  For the abelianisation $(\varphi|_K)_{ab}$ of $\varphi|_K$ on $H$, $i \in I$ and $j\in\mathbb Z$, we have
\begin{align}\label{E:phiKab}
(\varphi|_K)_{ab}(A_{i,j})=A_{\sigma(i),j+h(w_i)}.
\end{align}

\begin{proof}[Proof of claim:] Note that we have $\overline{x_{i_0}^r}A_{i,j}\overline{x_{i_0}^{-r}}=A_{i,j+r}$ and that for $k \neq i_0$ we have $\overline{x_k} = A_{k,0}$ and therefore $\overline{x_{k}^r}A_{i,j}\overline{x_{k}^{-r}}=A_{i,j}$ since $H$ is abelian.  Therefore for $\alpha\in F$, we have $\overline{\alpha} A_{i,j}\overline{\alpha^{-1}}=A_{i,j+h(\alpha)}$. Now for the abelianisation $(\varphi|_K)_{ab}$ of $\varphi|_K$ on $H$, $i\in I\backslash\{i_0\}$ and $j\in\mathbb Z$, we have
\begin{align*}
(\varphi|_K)_{ab}(A_{i,j})&=\overline{\varphi(x_{i_0}^jx_ix_{i_0}^{-j})}\\
&=\overline{(w_{i_0}x_{i_0}w_{i_0}^{-1})^jw_i}A_{\sigma(i),0}\overline{w_i^{-1}(w_{i_0}x_{i_0}w_{i_0}^{-1})^{-j}}\\
&=A_{\sigma(i),j+h(w_i)}.
\end{align*}
\end{proof}
\item For each infinite orbit $\mathcal O$ of $\sigma$, we define $V_{\mathcal O,t}\in H$ for $t\in\mathcal O\times \mathbb Z$ as following:
\begin{align*}
V_{\mathcal O,t}=V_{\mathcal O,(i,j)}=A_{i,j} \text{ for } t=(i,j) \in\mathcal O\times \mathbb Z.
\end{align*} 
We have, for $t=(i,j)\in\mathcal O\times \mathbb Z$, 
\begin{align}\label{E:psiabVOt}
(\psi|_K)_{ab}(V_{\mathcal O, t})&=(\psi|_K)_{ab}(A_{i,j})=A_{i,j+1}=V_{\mathcal O, (i,j+1)},\\
\label{E:phiabVOt}
(\varphi|_K)_{ab}(V_{\mathcal O, t})&=(\varphi|_K)_{ab}(A_{i,j})=A_{\sigma(i),j+h(w_i)}=V_{\mathcal O, (\sigma(i), j+h(w_i))}.
\end{align}
\item  For each finite orbit $\mathcal O$ which is not equal to the orbit $\{i_0\}$, we define $V_{\mathcal O,t}\in H$ for $t\in\mathbb Z$ as follows:


We fix a tuple $(k_1,\dotsc,k_r)$ such that $\mathcal O=\{k_1,\dotsc,k_r\}$ with $r=|\mathcal O|$, $\sigma(k_i)=k_{i+1}$ for $1\leq i\leq r-1$ and $\sigma(k_r)=k_1$. Let $h_{\mathcal O}=\sum_{i=1}^rh(w_{k_i})$, then by assumption $\gcd(h_{\mathcal O},r)=1$. Let $y_i=\sum_{1\leq j\leq i-1}-h(w_{k_j})$ for $i=1,2,\dotsc,r$. (Then we have, for example, $y_1=0$, $y_r-h(w_{k_r})=-h_{\mathcal O}$.) Since $\gcd(h_{\mathcal O},r)=1$, given  $t\in\mathbb Z$, there is a unique $i=i(t)\in\{1,2,\dotsc,r\}$ and a unique $j=j(t)\in\mathbb Z$ such that $t=r(y_i+j)+ih_{\mathcal O}$. Define a map $k:\mathbb Z\to\mathcal O$ by $k(t)=k_{i(t)}$. Now let
\begin{align}\label{E:DefinitionVOt}
V_{\mathcal O,t}=V_{\mathcal O,r(y_i+j)+ih_{\mathcal O}}:=A_{{k_i},j}=A_{{k_{i(t)},j(t)}}=A_{k(t),j(t)}\in H.
\end{align}
 Note that the map  $E_{\mathcal O}:\mathbb Z\to \mathcal O\times \mathbb Z$ given by  $t\mapsto (k(t),j(t))$ is a bijection.
 \item  \underline{Claim:} For a finite orbit $\mathcal O$ we have, for $t\in\mathbb Z$,
  \begin{align}\label{E:TheTranslationFormulaForPsi}
 &(\psi|_K)_{ab}(V_{\mathcal O,t})=V_{\mathcal O,t+|\mathcal O|}, \\ \label{E:TheTranslationFormula}
 &(\varphi|_K)_{ab}(V_{\mathcal O,t})=V_{\mathcal O,t+h_{\mathcal O}},
 \end{align} where $h_{\mathcal O}=\sum_{i\in\mathcal O}h(w_i)$. 

\begin{proof}[Proof of claim] 
Fix an arbitrary $t\in\mathbb Z$, we have $t=r(y_i+j)+ih_{\mathcal O}$ where $y_i$, $i=i(t)$ and $j=j(t)$ were defined in (3). By \eqref{E:DefinitionVOt}, we have $V_{\mathcal O,t}=A_{k_i,j}$. 

Thus we find
\begin{align*}
&(\psi|_K)_{ab}(V_{\mathcal O,t})=\psi(A_{k_i,j})=A_{k_i,j+1}=V_{\mathcal O,t+r}=V_{\mathcal O,t+|\mathcal O|},
\end{align*} proving \eqref{E:TheTranslationFormulaForPsi}.

To prove \eqref{E:TheTranslationFormula}, we note that
\begin{align}\label{E:PhiKab}
&(\varphi|_K)_{ab}(V_{\mathcal O,t})=(\varphi|_K)_{ab}(A_{k_i,j}).
\end{align}
We have the following, by \eqref{E:phiKab} and recalling the definition of $y_i$:
\begin{enumerate}
\item If $1\leq i(t)\leq r-1$, \eqref{E:PhiKab} is equal to
\begin{align*}
A_{k_{i+1},j+h(w_{k_i})}=V_{\mathcal O,r(j+h(w_{k_i})+y_{i+1})+(i+1)h_{\mathcal O}}=V_{\mathcal O,t+h_{\mathcal O}}
\end{align*}
as desired.

\item If $i(t)=r$, \eqref{E:PhiKab} is equal to
\begin{align*}
A_{k_{1},j+h(w_{k_r})}=V_{\mathcal O,r(j+h(w_{k_r})+y_1)+h_{\mathcal O}}=V_{\mathcal O,t+h_{\mathcal O}}
\end{align*} as desired.
\end{enumerate}
\end{proof}
\item Define an ordered basis of $H$.

Let $\mathbb O(\sigma)$ be the set of orbits of $\sigma$ which are not equal to $\{i_0\}$. For a finite orbit $\mathcal O$, we let $\mu(\mathcal O)=\mathbb Z$. For an infinite orbit $\mathcal O$, we let $\mu(\mathcal O)=\mathcal O\times \mathbb Z$. Then $$\mathfrak{B}':=\{V_{\mathcal O,t}|\mathcal{O} \in \mathbb O(\sigma), t\in\mu(\mathcal O)\}$$ is a basis of $H$. To see this, let us check that $\mathfrak{B}'=\mathfrak{B}$, the latter being the basis of $H$ given by \eqref{E:ABaseOfH}. First, it is easy to check that $\mathfrak{B}'\subseteq\mathfrak{B}$ by definition. Now, let $A_{i,j}\in\mathfrak B$ and we want to show $A_{i,j}\in\mathfrak B'$. Consider two cases. For the first case, suppose $i$ is contained in an infinite orbit $\mathcal O$. Then $(i,j)\in\mu(\mathcal O)=\mathcal O\times \mathbb Z$ and thus $A_{i,j}=V_{\mathcal O,(i,j)}$ is in $\mathfrak{B}'$. For the second case, suppose $i$ is contained in a finite orbit $\mathcal O$ distinct from $\{i_0\}$. 
Recall that $E_{\mathcal O}(t)=(k(t),j(t))$ is a bijection from $\mathbb Z$ to $\mathcal O\times \mathbb Z$. Hence there is a unique $t\in\mathbb Z$, such that $k(t)=i$ and $j(t)=j$. Now we have $V_{\mathcal O,t}=A_{k(t),j(t)}=A_{i,j}$, thus $A_{i,j}$ is contained in $\mathfrak B'$. 


We now define a total order $\prec$ on the index set $$\mathcal I:=\{(\mathcal O,t)\mid \mathcal O\in \mathbb O(\sigma), t\in\mu(\mathcal O)\}.$$
First, let $<'$ be an arbitrary total order of $\mathbb O(\sigma)$. Second, we define an ordering $<$ on each $\mu(\mathcal O)$. If $\mathcal O$ is finite, $<$ is the usual order on $\mu(\mathcal O)=\mathbb Z$. When $\mathcal O$ is infinite, let $(i,j), (i',j')\in \mu(\mathcal O)=\mathcal O\times \mathbb Z$. We write  $(i,j)<(i',j')$ if and only if  $i'\in\{\sigma^k(i)\mid k\geq1\}$ or $(i',j')\in\{(i,j+k)\mid k\geq1\}$.  

 The order $\prec$ is defined lexicographically: $(\mathcal O, t)\prec (\mathcal O',t')$ if and only if $\mathcal O<'\mathcal O'$ or $\mathcal O=\mathcal O'$ and $t<t'$ in $\mu(\mathcal O)$. Now set $\mathcal B=(\mathfrak{B}',\prec)$, which is an ordered basis of $H$.

\item Last we check that both $(\varphi|_K)_{ab}$ and $(\psi|_K)_{ab}$ are  positively triangular with respect to $\mathcal B$.

This essentially follows from \eqref{E:psiabVOt}, \eqref{E:phiabVOt}, \eqref{E:TheTranslationFormulaForPsi} and \eqref{E:TheTranslationFormula}. The details are as follows. 

First, define $\eta:\mathcal I\rightarrow\mathcal I$ by
\begin{align*}
\eta(\mathcal O,t)= 
\begin{cases}(\mathcal O, (\sigma(i),j+h(w_i)))& \text{ if }|\mathcal O|=\infty \text{ and } t=(i,j)\in\mathcal O\times\mathbb Z\\
             (\mathcal O, t+h_{\mathcal O})& \text{ if }|\mathcal O|<\infty \text{ and } t\in\mathbb Z.
\end{cases}
\end{align*}
We can then verify that $\eta$ preserves the order $\prec$. Therefore $(\varphi|_K)_{ab}$ is positively triangular with respect to $\mathcal B$ by \eqref{E:phiabVOt} and \eqref{E:TheTranslationFormula}. 

Second, define $\zeta:\mathcal I\to\mathcal I$ by 
\begin{align*}
\zeta(\mathcal O,t)=
\begin{cases}(\mathcal O, (i,j+1))& \text{ if }|\mathcal O|=\infty \text{ and } t=(i,j)\in\mathcal O\times\mathbb Z\\
             (\mathcal O, t+|\mathcal O|)& \text{ if }|\mathcal O|<\infty \text{ and } t\in\mathbb Z.
\end{cases}
\end{align*}
We can verify that $\zeta$ preserves the order $\prec$. Therefore $\psi|_K$ is positively triangular with respect to $\mathcal B$ by \eqref{E:psiabVOt} and \eqref{E:TheTranslationFormulaForPsi}. 
 \end{enumerate}
\end{proof}

Theorem \ref{T:CoprimeToOrderpreserving} says that if an automorphism as in the statement of the theorem satisfies the coprime condition: $\gcd(h_\mathcal O,|\mathcal O|)=1$ for each finite orbit $\mathcal O$ of $\sigma$, then the automorphism preserves a bi-ordering of the free group. 
If we replace this coprime condition by the (weaker) non-vanishing condition: $h_{\mathcal O}\neq0$ for each finite orbit of $\sigma$ with $|\mathcal O|\geq2$, we can still prove that the automorphism preserves a left-ordering.

The proof of this fact remains almost the same; the major change is the definition of $V_{\mathcal O,t}\in H$ for $t\in\mathbb Z$ and each finite orbit $\mathcal O$ of $\sigma$, which is given by the following. Assume that $\mathcal O=\{k_1,\dotsc,k_r\}$ with $r=|\mathcal O|$, $\sigma(k_i)=k_{i+1}$ for $1\leq i\leq r-1$. If $r=1$, we just define $V_{\mathcal O,t}=A_{k_1,t}\in H$, where $t\in\mathbb Z$.
Now consider $r\geq2$. For each $t\in\mathbb Z$, there are unique integers $i=i(t), j=j(t), m=m(t)$ (by the Euclidean algorithm) such that
 $$t=(mr+j)h_{\mathcal O}+i\text{ with } 0\leq i\leq|h_{\mathcal O}|-1, 1\leq j\leq r, m\in\mathbb Z.$$
  Then we define 
  $$V_{\mathcal O,t}=A_{{k_{j(t)}},m(t)h_{\mathcal O}+h_{j(t)}+i(t)}, \quad\forall t\in\mathbb Z,$$ where $h_j:=\sum_{1\leq p\leq j-1}h(w_{k_p})$ for $j=1,2,\dotsc,r$.  We then verify that the constructed bi-ordering on $K$ is still invariant under $\phi$ and that the left-ordering from the exact sequence is invariant under $\varphi$.
  
  \begin{example}
Consider the free group $F$ with basis $\{x_1, x_2, x_3\}$ and $\varphi:F \rightarrow F$ given by $x_1\mapsto x_3x_2x_3^{-1}, x_2\mapsto x_3x_1x_3^{-1}, x_3\mapsto x_3$.  With $i_0=3$ we see $h_{\mathcal O}=2$ for the orbit $\mathcal O=\{1,2\}$, and so $\varphi$ preserves a left-ordering, although it clearly does not preserve a bi-ordering. 
\end{example}

\begin{example}\label{E:OrderPreservingButNotPositivelyTriangular}
Theorem \ref{T:CoprimeToOrderpreserving} gives us an example of automorphism $\varphi$ on free group $F$ which is order preserving but $\varphi_{ab}\otimes Id$ is not positively triangular on $F/[F,F]\otimes\mathbb R$. For instance, let $F$ be free with basis $\{x_1, x_2, x_3\}$ and let $\varphi$ be the automorphism of $F$ defined by
$$ x_1\mapsto x_3x_2x_3^{-1}, x_2\mapsto x_1, x_3\mapsto x_3.$$ Then $\varphi$ is order-preserving by Theorem \ref{T:CoprimeToOrderpreserving}. However $\varphi_{ab}\otimes Id$ is not positively triangular on $F/[F,F]\otimes\mathbb R$.

To see this, suppose that $\varphi_{ab}\otimes Id:F/[F,F]\otimes\mathbb R \rightarrow F/[F,F]\otimes\mathbb R$ is positively triangular, then by Lemma \ref{L:PositiveTriangularbi-ordering} $\varphi_{ab}\otimes Id$ preserves a bi-ordering on $F/[F,F]\otimes\mathbb R$.   But this is impossible as $\varphi_{ab}\otimes Id$ swaps the two elements $\overline{x_1}\otimes 1$ and $\overline{x_2}\otimes 1$.
\end{example}

\section{Order-preserving braids, bi-orderable link groups and 3-manifolds}
\label{section:links}

Recall from the introduction that the braid group $B_n$ is generated by $\sigma_1,\dotsc,\sigma_{n-1}$ subject to $\sigma_i\sigma_{i+1}\sigma_i=\sigma_{i+1}\sigma_i\sigma_{i+1}$ for  $1\leq i\leq n-2$, and 
$\sigma_i\sigma_j=\sigma_j\sigma_i$ for $|i-j|>1$; and that the Artin action of $B_n$ on $F_n$ is determined by the following action of the generators $\sigma_i$ of $B_n$ on the generators $x_j$ of $F_n$:
\begin{displaymath}
   x_j^{\sigma_i} = \left\{
     \begin{array}{ll}
        x_ix_{i+1}x_i^{-1} & \mbox{ if $j = i$} \\
       x_i & \mbox{ if $j = i+1$} \\
        x_j & \mbox{ if $j \neq i, i+1$.}
     \end{array}
   \right.
\end{displaymath} 

We will call a braid $\beta \in B_n$ order-preserving if its action on $F_n$ induces an order-preserving automorphism $F \rightarrow F$.

Every braid has an associated permutation, which is its image under the homomorphism $B_n \rightarrow S_n$ given by sending $\sigma_i$ to the transposition that exchanges $i$ and $i+1$. We will often refer to the image of $\beta$ under this map as ``\textbf{the underlying permutation} of $\beta$."  The kernel of this homomorphism is $PB_n$, the subgroup of pure braids.

As our convention is that the action of $B_n$ on $F_n$ is a right action (consistent with \cite{KR18}), for this section our convention for composition of permutations is to do the leftmost first.  For example, $\sigma \tau$ acts on elements first by $\sigma$. We therefore write the action of a permutation $\sigma$ on an element $i$ as exponentiation, $i^{\sigma}$.

\subsection{Minimal volume 3-manifolds}

In \cite{KR18}, the authors initiated a study of bi-orderability of the fundamental groups of minimal volume cusped hyperbolic $3$-manifolds.  For such manifolds, they determined whether or not the fundamental group is bi-orderable for all examples with five or fewer cusps, with one exception.  The fundamental group of a three-cusped manifold, called the ``magic manifold", could not be addressed by the techniques of \cite{KR18}. 

\begin{proposition}
\label{P:magic}
    The fundamental group of the magic manifold is bi-orderable.
\end{proposition}
\begin{proof}
By \cite[Section 5.5]{KR18}, the fundamental group of the magic manifold is bi-orderable if and only if the braid $\sigma_1^2 \sigma_2^{-1}$ is order-preserving. Let $F_3$ denote the free group on generators $x_1, x_2, x_3$.  One can verify that $\sigma_1^2 \sigma_2^{-1}$ induces a map $\varphi: F_3 \rightarrow F_3$ given by 
\[\varphi(x_1) = x_1x_3x_1x_3^{-1}x_1^{-1}, \, \varphi(x_2) = x_1x_3x_1^{-1}\, \varphi(x_3) = x_3^{-1}x_2x_3,
\]
the underlying permutation $\sigma$ is the transposition $(2 \; 3)$. 

In order to apply Theorem \ref{T:CoprimeToOrderpreserving}, we define $h: F_3 \rightarrow \mathbb{Z}$ by $h(x_1) = 1, h(x_2) = 0, h(x_3) = 0$.  Then observe that $\sigma$ has a single orbit $\mathcal O = \{2,3\}$ that does not contain $1$, and that $h_{\mathcal O} = h(x_1)+h(x_3^{-1}) = 1$.  Since $\gcd(2,1) = 1$, the braid $\sigma_1^2 \sigma_2^{-1}$ is order-preserving.
\end{proof}

This resolves \cite[Question 5.6]{KR18} in the affirmative. Moreover, from the proof, we see that the braid $\beta = \sigma_1^2\sigma_2^{-1}$ is order-preserving. Note that multiplication of this by $\sigma_2^2$ does not change the underlying permutation and also does not alter the value of $h_{\mathcal O}$ in the above discussion.  The same is true of multiplication by $\sigma_2^{-2}$.  Therefore the braid 
$\sigma_1^2\sigma_2^{2n-1}$ is also order-preserving.  On the other hand $\sigma_1^2\sigma_2^{2n}$, being a pure braid, is also order-preserving.  Therefore we have the following.

\begin{theorem}
For every integer $n$, the braid $\sigma_1^2\sigma_2^{n}$ is order-preserving.
\end{theorem} 

We note that it was shown by Johnson, Scherich and Turner that the braids $\sigma_1\sigma_2^{2k+1}$
are not order-preserving for any integer $k$ \cite{JST23}.  Such 3-braids have underlying permutation a single cycle.  We do not know if there exists an order-preserving braid whose permutation is a single cycle. 

\subsection{Dehn surgery and sublinks of bi-orderable links}

We can also improve upon \cite[Theorem 5.12]{KR18}, which says that every link $L$ in $S^3$ is a sublink of a link $L'$ such that $\pi_1(S^3 \setminus L')$ is bi-orderable.  We provide a construction of such an $L'$ by adding only two components, and show that a similar result holds for all links in arbitrary $3$-manifolds.

Recall that the pure braid group $PB_n$ is generated by 
\[ A_{i,j} = \sigma_{j-1} \dots \sigma_{i+1} \sigma_i^2 \sigma_{i+1}^{-1} \dots \sigma_{j-1}^{-1}
\]
where $1 \leq i < j \leq n$ \cite[Lemma 1.8.2]{birman74}.  For example, we have $A_{i,i+1}=\sigma_i^2$ for $1\leq i\leq n-1$ and $A_{i,i+2}=\sigma_{i+1}\sigma_i\sigma_{i+1}^{-1}$ for $1\leq i\leq n-2$.

\begin{lemma}
\label{L:Aijcomputation}
Fix $i,j$ with $1 \leq i<j \leq n$.  Then we have the following:
\begin{displaymath}
   x_k^{A_{i,j}} = \left\{
     \begin{array}{ll}
        (x_ix_j)x_i(x_j^{-1}x_i^{-1}) & \mbox{ if $k = i$} \\
       (x_ix_jx_i^{-1}x_j^{-1})x_k(x_jx_ix_j^{-1}x_i^{-1}) & \mbox{ if $i+1\leq k\leq j-1$} \\
       x_ix_jx_i^{-1} & \mbox{ if $k = j$}\\
        x_k & \mbox{ otherwise.}
     \end{array}
   \right.
\end{displaymath} 

\end{lemma}
\begin{proof}
First, if $k \notin \{i,i+1, \dots, j\}$ then $ x_k^{A_{i,j}} = x_k$ follows from the observation that $x_k^{\sigma_{\ell}} = x_k$ for all $\ell \in \{i,i+1, \dots, j-1\}$.

Next consider $ x_i^{A_{i,j}}$.  We compute
\begin{align*}
x_i^{A_{i,j}}=&x_i^{\sigma_i^2 \sigma_{i+1}^{-1} \dots \sigma_{j-1}^{-1}}\\
=&(x_ix_{i+1}x_i^{-1}x_ix_ix_{i+1}^{-1}x_i^{-1})^{\sigma_{i+1}^{-1} \dots \sigma_{j-1}^{-1}}\\
=& (x_ix_j)x_i(x_j^{-1}x_i^{-1}).
\end{align*}
One similarly computes that $x_j^{A_{ij}}=x_ix_jx_i^{-1}$.

On the other hand, if $k \in \{i+1, \dots, j-1\}$ 
then
\begin{align*}
x_k^{A_{i,j}}=&x_k^{\sigma_k \dots \sigma_{i+1}\sigma_i^2 \sigma_{i+1}^{-1} \dots \sigma_{j-1}^{-1}}\\
=& (x_kx_{k+1}x_k^{-1})^{\sigma_{k-1} \dots \sigma_{i+1}\sigma_i^2 \sigma_{i+1}^{-1} \dots \sigma_{j-1}^{-1}}\\
=& (x_ix_{k+1}x_i^{-1})^{\sigma_i \sigma_{i+1}^{-1} \dots \sigma_{j-1}^{-1}}\\
=& ((x_ix_{i+1}x_i^{-1})x_{k+1}(x_ix_{i+1}^{-1}x_i^{-1}))^{\sigma_{i+1}^{-1} \dots \sigma_{j-1}^{-1}}\\
=& ((x_ix_{k}x_i^{-1})x_{k+1}(x_ix_{k}^{-1}x_i^{-1}))^{\sigma_{k}^{-1} \dots \sigma_{j-1}^{-1}}\\
=& ((x_ix_{k+1}x_i^{-1}x_{k+1}^{-1})x_k(x_{k+1}x_ix_{k+1}^{-1}x_i^{-1}))^{\sigma_{k+1}^{-1} \dots \sigma_{j-1}^{-1}}\\
=& (x_ix_jx_i^{-1}x_j^{-1})x_k(x_jx_ix_j^{-1}x_i^{-1}).
\end{align*}

\end{proof}


\begin{lemma}
\label{L:changehc}
    Let $\beta \in B_{n}$ 
    with underlying permutation $\sigma$ and fix $i$ with $1\leq i<n$.  
    Suppose that $\beta$ yields the automorphism $\phi:F_n \rightarrow F_n$ with $\phi(x_k) = w_k x_{j} w_k^{-1}$ where $j = k^\sigma$, and 
    let $\psi :F_n \rightarrow F_n$ given by $\psi(x_k) = u_k x_{j}u_k^{-1}$ denote the automorphism of $F_n$ induced by $\beta A_{i,n}$.
    Define homomorphism $h : F_n \rightarrow \mathbb{Z}$ by $h(x_n) =1$ and $h(x_k) = 0$ for all $k=1, \ldots, n-1$.  For each cycle $c = (k_1, \; \cdots \; ,k_r)$ in the cycle decomposition of $\sigma$, set $h_c = \sum_{i=1}^rh(w_{k_i})$ and $\ell_c =\sum_{i=1}^rh(u_{k_i})$.  Then $\ell_c = h_c+1$ if $i\in \{ k_1, \dots, k_r\}$ and $\ell_c = h_c$ otherwise.
\end{lemma}
\begin{proof}
    Given $k \in \{1, \dots, n\}$ we write $j=k^\sigma$ and compute that 
    $u_k x_{j}u_k^{-1}= \psi(x_k) = x_k^{\beta A_{i,n}} = (w_k x_{j} w_k^{-1})^{A_{i,n}}$. Thus we have
    \begin{displaymath}
   u_k = \left\{
     \begin{array}{ll}
        w_k^{A_{i,n}}x_ix_n & \mbox{ if $j = i$} \\
        w_k^{A_{i,n}}(x_ix_nx_i^{-1}x_n^{-1}) & \mbox{ if $j \in \{i+1, \dots, n-1\}$} \\
        w_k^{A_{i,n}}x_i &
        \mbox{ if $j=n$}\\
        w_k^{A_{i,n}} & \mbox{ otherwise,}
     \end{array}
   \right.
\end{displaymath} 
by appealing to Lemma \ref{L:Aijcomputation}.

Note that by Lemma \ref{L:Aijcomputation} we have $h(w_k^{A_{i,n}}) = h(w_k)$, and so we conclude that $h(u_k) = h(w_k)+1$ if $k^\sigma = i$, and $h(u_k) = h(w_k)$ otherwise.  The conclusion of the lemma follows immediately from this observation.
\end{proof}

In the proposition below, for $g_1, \ldots , g_n \in G$ we use the notation $sg\{g_1, \dots ,g_n\}$ to denote the subsemigroup of $G$  generated by $\{g_1, \dots ,g_n\}$.

\begin{proposition}
\label{P:orderpres}
Suppose $n \geq 3$.  For all $ \beta \in B_{n-1}$ there exists $\alpha \in sg\{A_{1, n}, \dots, A_{n-1, n}\}$ such that $\beta \alpha$ is order-preserving. 
\end{proposition}
\begin{proof}
Let $\sigma$ denote the underlying permutation of the braid $\beta$, and observe that $n^\sigma=n$.  Since every $\alpha \in sg\{A_{1, n}, \dots, A_{n-1, n}\}$ is a pure braid, the underlying permutation of the braid $\beta \alpha$ is also $\sigma$ and so satisfies $n^\sigma=n$.  So, our task is to choose a braid $\alpha$ such that the relative primeness condition of Theorem \ref{T:CoprimeToOrderpreserving} is satisfied by $\beta \alpha$. 

For this, we inductively apply Lemma \ref{L:changehc},
right-multiplying $\beta$ by braids $A_{i,n}$ as needed to guarantee that the relative primeness condition holds.
\end{proof}

\begin{example}
As an example of how to construct $\alpha$ as in Proposition \ref{P:orderpres}, consider the braid $\beta = \sigma_1 \sigma_3 \in B_5$ as in Figure \ref{fig:alphaexample}. Since the generators $\sigma_i$ are not order-preserving \cite{KR18}, it follows that $\beta$ is not order-preserving by \cite[Corollary 9]{rolfsen18}. In particular, the underlying permutation of $\beta$ is $(1 \; 2) (3 \; 4)$, and each cycle fails the coprime condition of Theorem \ref{T:CoprimeToOrderpreserving} with $i_0=5$.  Yet, as we see in Figure \ref{fig:alphaexample}, we can multiply the braid $\beta$ by $\alpha = A_{4,5} A_{2,5}$ in order to produce a braid in $B_5$ that \emph{is} order-preserving.  The factors of $A_{4,5}$ and $A_{2,5}$ guarantee that the comprime condition holds for the cycles $(3 \; 4)$ and  $(1 \; 2)$ respectively.

\begin{figure}
    \centering
    \includegraphics{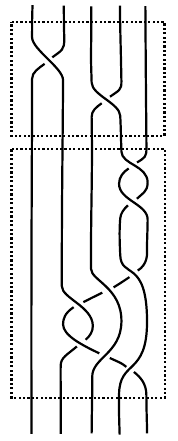}
    \caption{The braid $\beta = \sigma_1 \sigma_3$ in the dashed box on top, and the braid $\alpha = A_{4,5} A_{2,5}$ in the dashed box on the bottom.}
    \label{fig:alphaexample}
\end{figure}
\end{example}

In fact, the previous proposition holds for every braid $\beta \in B_n$ whose underlying permutation fixes $n$.  However, the statement of Proposition \ref{P:orderpres} is sufficient for our next theorem.

\begin{theorem}
\label{T:2morecomponents}
    Every $n$-component link $L$ in $S^3$ is a sublink of an $(n+2)$-component link $L'$ in $S^3$ such that $\pi_1(S^3 \setminus L')$ is bi-orderable.
\end{theorem}
\begin{proof}
By Alexander's theorem \cite{Alexander23}, the $n$-component link $L$ can be written as the closure of a braid $\beta \in B_k$ for some $k$.  By Proposition \ref{P:orderpres}, there is a braid $\alpha \in sg\{A_{1, k+1}, \dots, A_{k, k+1}\} \subset B_{k+1}$ such that $\beta \alpha$ is order-preserving.  The closure of $\beta \alpha$, together with the braid axis, gives a link $L'$ having $n+2$ components such that $\pi_1(S^3 \setminus L')$ is bi-orderable.
\end{proof}

\begin{corollary}
    Every link $L$ in a compact, connected, closed orientable $3$-manifold $M$ is a sublink of a link $L'$ in $M$ such that $\pi_1(M \setminus L')$ is bi-orderable. 
 Moreover, if $n$ denotes the minimal number of surgery curves in in $S^3$ needed to produce $M$, then $L'$ is obtained from $L$ by adding $n+2$ components.  
\end{corollary}

\begin{proof}
By the Lickorish-Wallace theorem, $M$ can be obtained by surgery on $S^3$ \cite{lickorish, wallace}, so there are $n$-component links $L_1 \subset M$ and $L_2 \subset S^3$ and a homeomorphism $h: S^3 \setminus L_2 \to M \setminus L_1$.  We may assume that $L_1$ and $L$ are disjoint.  By Theorem \ref{T:2morecomponents}, there is a link $L_3 \subset S^3$ with $ L_2 \cup h^{-1}(L) \subset L_3$ and $\pi_1(S^3 \setminus L_3)$ bi-orderable; moreover, $L_3$ has two more components than $L_2 \cup h^{-1}(L)$.  Define $L' = h(L_3 \setminus L_2) \cup L_1$.  Then 
$h(S^3 \setminus L_3) = M \setminus L'$, so $\pi_1(M \setminus L')$ is bi-orderable.    
\end{proof}


\subsection{Multiplication by braids in \texorpdfstring{$B_{n-1}$}{B(n-1)}}

We can also generalise the observations made following Proposition \ref{P:magic}. There, we observed that it is possible to change a braid in $B_3$ which is \emph{not} order preserving into an order-preserving braid by multiplying by powers of a single generator of $B_3$.  

\begin{proposition}\label{P:B3TimsElementsInB2}
Given $\beta \in B_3$, there exists $k$ with $0 \leq k \leq 3$ such that $\beta \sigma_1^k$ is order-preserving.
\begin{proof}
    Let $\beta \in B_3$ be given.  If $\beta$ is pure then $k=0$ suffices, and if the underlying permutation of $\beta$ is $(1 \; 2)$, then $\beta \sigma_1$ is pure, and therefore $k=1$ suffices.  

On the other hand, suppose the underlying permutation of $\beta$ is $\tau = (1 \; 3)$, and define $h:F_3 \rightarrow \mathbb{Z}$ by $h(x_1) = h(x_3) =0$ and $h(x_2) =1$.  If $\beta$ yields an automorphism $\phi :F_3 \rightarrow F_3$ given by 
\[ \phi(x_1) = w_1 x_{3} w_1^{-1}, \; \phi(x_2) = w_2 x_{2} w_2^{-1}, \; \phi(x_3) = w_3 x_{1} w_3^{-1}
\]
and $h(w_1) + h(w_3)$ is odd, then $\beta$ is order-preserving by Theorem \ref{T:CoprimeToOrderpreserving}.  On the other hand if $h(w_1) + h(w_3)$ is even, then consider the braid $\beta \sigma_1^2$, also with underlying permutation $(1 \; 3)$ which acts by the permutation 
\[ \psi(x_1) = w_1^{\sigma_1^2} x_{3} (w_1^{-1})^{\sigma_1^2}, \; \psi(x_2) = w_2^{\sigma_1^2} x_1x_2x_1^{-1} (w_2^{-1})^{\sigma_1^2}, \; \psi(x_3) = w_3^{\sigma_1^2} x_1x_2x_1x_2^{-1}x_1^{-1} (w_3^{-1})^{\sigma_1^2}.
\]
We see that $h(w_1^{\sigma_1^2})+h(w_3^{\sigma_1^2} x_1x_2) = h(w_1) + h(w_3) +1$ is odd, and so $\beta \sigma_1^2$ satisfies the coprime condition of Theorem \ref{T:CoprimeToOrderpreserving} and so is order-preserving.  When the underlying permutation is $(2 \; 3)$, we similarly prove that $\beta$ or $\beta\sigma_1^2$ is order-preserving. 

Last, if the underlying permutation of $\beta$ is a $3$-cycle, then the braid $\beta \sigma_1$ has underlying permutation a single transposition. Then $(\beta \sigma_1)\sigma_1^k$ is order-preserving by the previous paragraphs, for some $k$ with $0 \leq k \leq 2$.
\end{proof}

\end{proposition}

We can deal with $n \geq 4$ by a similar, general argument.

\begin{lemma} 
\label{permutation lemma}
    Suppose that $\sigma \in S_n$ is a permutation with $n^\sigma \neq n$ where $n \geq 4$.  Then there exists $\tau \in S_{n-1}$ such that $\sigma \tau$ has disjoint cycle decomposition $c_1 c_2$, where $c_1 = (i)$ for some $i \in \{1, \dots, n-2\}$ and $c_2$ is an $(n-1)$-cycle.
\end{lemma}
\begin{proof}
To complete the proof, we need to find two disjoint cycles $c_1$ and $c_2$ with $c_1=(i)$ for $1\leq i<n-2$ and $c_2$ a $(n-1)$-cycle in $S_n$, such that $\tau:=\sigma^{-1}c_1c_2$ fixes $n$. Assume that $n^{\sigma^{-1}}=j\neq n$. Arrange the elements of the set $\{1,2,\dotsc,n\}\backslash\{ j, n\}$ in any order $i_1,i_2,\dotsc,i_{n-2}$ with $i_{n-2} \neq n-1$. Now $c_1=(i_{n-2})$ and $c_2=(j \; n  \;  i_1 \; i_2 \; \dotsc \; i_{n-3})$ are the desired cycles.
\end{proof}

\begin{proposition}
    If $n \geq 3$ and $\beta \in B_n$, then there exists $\alpha \in B_{n-1}$ such that $\beta\alpha $ is order-preserving. 
\end{proposition}

\begin{proof}
The case $n=3$ is given by Proposition \ref{P:B3TimsElementsInB2}, thus we assume $n \geq 4$.  Suppose that $\beta$ has underlying permutation $\sigma$ and that $\sigma(n) = n$.  Then we can choose $\alpha \in B_{n-1}$ such that the underlying permutation of $\beta \alpha$ is trivial.  Therefore $\beta \alpha$ is a pure braid, and so is order-preserving.

Now consider the case that  $\sigma(n) \neq n$. Choose $\tau \in S_{n-1}$ such that the cycle decomposition of $\sigma \tau$ is $c_1c_2$ where $c_1 = (i_0)$ with $i_0 < n-1$ and $c_2$ is an $(n-1)$-cycle, which we can do by Lemma \ref{permutation lemma}.  Choose a braid $\alpha \in B_{n-1}$ having underlying permutation $\tau$.

Suppose that $\beta \alpha$ yields an automorphism $\phi :F_n \rightarrow F_n$ given by 
\[ \phi(x_i) = w_i x_{j} w_i^{-1}
\]
where $j = i^{\sigma \tau}$.  Define $h:F_n \rightarrow \mathbb{Z}$ by $h(x_{i_0}) = 1$ and $h(x_i) = 0$ for all $i\neq i_0$ and set $h_{c_2} = \sum_{i \neq i_0} h(w_i)$.  If $\gcd(h_{c_2}, n-1) = 1$ then $\beta \alpha$ is order-preserving by Theorem \ref{T:CoprimeToOrderpreserving}.  

If $\gcd(h_{c_2}, n-1) \neq 1$, then consider the product $\beta \alpha \sigma_{i_0}^2$.  This braid gives rise to an automorphism $\psi :F_n \rightarrow F_n$ given by 
\[ \psi(x_i) = v_i x_{j} v_i^{-1}
\]
where $j = i^{\sigma \tau}$, as the underlying permutation is unchanged.  We compute:

\begin{displaymath}
   (w_i x_{j} w_i^{-1})^{\sigma_{i_0}^2} = \left\{
     \begin{array}{ll}
         (w_i)^{\sigma_{i_0}^2} x_{i_0}x_{i_0+1}x_{i_0}x_{i_0+1}^{-1}x_{i_0}^{-1} (w_i^{-1})^{\sigma_{i_0}^2}& \mbox{ if $j = i_0$} \\
       (w_i)^{\sigma_{i_0}^2} x_{i_0} x_{i_0+1}x_{i_0}^{-1}(w_i^{-1})^{\sigma_{i_0}^2} & \mbox{ if $j = i_0+1$} \\
        (w_i)^{\sigma_{i_0}^2} x_{j} (w_i^{-1})^{\sigma_{i_0}^2} & \mbox{ if $j \neq i_0, i_0+1$.}
     \end{array}
   \right.
    \end{displaymath} 
Consequently $v_i = w_i^{\sigma_1^2}$ whenever $i^{\sigma \tau} \neq i_0, i_0+1$
and $v_i = w_i^{\sigma_{i_0}^2}x_{i_0}$ when $i^{\sigma \tau} = i_0+1$.  Therefore we can compute the quantity $h'_{c_2} = \sum_{i\neq i_0}^n h(v_i)$ by observing that $h_{c_2}' = 1+ \sum_{i \neq i_0}^n h(w_i^{\sigma_1^2}) = 1+ h_{c_2}$, where the final equality follows from observing that $h(w_i^{\sigma_1^2}) = h(w_i)$ for all $i$.  It follows that if  $\gcd(h'_{c_2}, n-1) = 1$, then $\beta \alpha \sigma_1^2$ is order-preserving by Theorem \ref{T:CoprimeToOrderpreserving}.

If $\gcd(h'_{c_2}, n-1) \neq 1$, we may continue inductively to add powers of $\sigma_{i_0}^2$, until we arrive at a braid $\beta \alpha \sigma_{i_0}^{2 \ell}$ which satisfies the coprime condition of Theorem \ref{T:CoprimeToOrderpreserving}, and is therefore order-preserving.
\end{proof}

\bibliography{order_preserving}

\begin{thebibliography}{10}

\bibitem{Alexander23}
James~W. Alexander.
\newblock A lemma on systems of knotted curves.
\newblock {\em Proc. Natl. Acad. Sci. U. S. A.}, 9(3):93--95, 1923.

\bibitem{Artin47}
E.~Artin.
\newblock Theory of braids.
\newblock {\em Ann. of Math. (2)}, 48:101--126, 1947.

\bibitem{birman74}
Joan~S. Birman.
\newblock {\em Braids, links, and mapping class groups}, volume No. 82 of {\em
  Annals of Mathematics Studies}.
\newblock Princeton University Press, Princeton, NJ; University of Tokyo Press,
  Tokyo, 1974.

\bibitem{CDN16}
Adam Clay, Colin Desmarais, and Patrick Naylor.
\newblock Testing bi-orderability of knot groups.
\newblock {\em Canad. Math. Bull.}, 59(3):472--482, 2016.

\bibitem{CR12}
Adam Clay and Dale Rolfsen.
\newblock Ordered groups, eigenvalues, knots, surgery and {$L$}-spaces.
\newblock {\em Math. Proc. Cambridge Philos. Soc.}, 152(1):115--129, 2012.

\bibitem{Fox53}
Ralph~H. Fox.
\newblock Free differential calculus. {I}. {D}erivation in the free group ring.
\newblock {\em Ann. of Math. (2)}, 57:547--560, 1953.

\bibitem{ito17}
Tetsuya Ito.
\newblock Alexander polynomial obstruction of bi-orderability for rationally
  homologically fibered knot groups.
\newblock {\em New York J. Math.}, 23:497--503, 2017.

\bibitem{JJ23}
Jonathan Johnson.
\newblock Residual torsion-free nilpotence, biorderability and pretzel knots.
\newblock {\em Algebr. Geom. Topol.}, 23(4):1787--1830, 2023.

\bibitem{Johnson_2024}
Jonathan Johnson.
\newblock Residual torsion-free nilpotence, bi-orderability, and two-bridge
  links.
\newblock {\em Canadian Journal of Mathematics}, 76(2):394–457, 2024.

\bibitem{JST23}
Jonathan Johnson, Nancy Scherich, and Hannah Turner.
\newblock Algorithmic obstructions and order-preserving braids.
\newblock Preprint.

\bibitem{KR18}
Eiko Kin and Dale Rolfsen.
\newblock Braids, orderings, and minimal volume cusped hyperbolic 3-manifolds.
\newblock {\em Groups Geom. Dyn.}, 12(3):961--1004, 2018.

\bibitem{lickorish}
W.~B.~R. Lickorish.
\newblock A representation of orientable combinatorial {$3$}-manifolds.
\newblock {\em Ann. of Math. (2)}, 76:531--540, 1962.

\bibitem{PR03}
Bernard Perron and Dale Rolfsen.
\newblock On orderability of fibred knot groups.
\newblock {\em Math. Proc. Cambridge Philos. Soc.}, 135(1):147--153, 2003.

\bibitem{rolfsen18}
Dale Rolfsen.
\newblock Ordered groups as a tensor category.
\newblock {\em Pacific J. Math.}, 294(1):181--194, 2018.

\bibitem{wallace}
Andrew~H. Wallace.
\newblock Modifications and cobounding manifolds.
\newblock {\em Canadian J. Math.}, 12:503--528, 1960.

\bibitem{yamada}
Takafumi Yamada.
\newblock A family of bi-orderable non-fibered 2-bridge knot groups.
\newblock {\em J. Knot Theory Ramifications}, 26(4):1750011, 8, 2017.

\end{thebibliography}

\bibliographystyle{plain}

\end{document}